\documentclass[12pt,leqno,twoside]{amsart}
\usepackage{amsmath, amscd, amssymb, amsfonts, tikz-cd, enumitem, appendix} 
\usepackage{latexsym}
\usepackage[all,cmtip]{xy}
\usepackage{graphicx}
\usepackage[colorlinks=true, pdfstartview=FitV, linkcolor=blue, citecolor=blue, urlcolor=blue]{hyperref}
\usepackage{subcaption} 
\usepackage{mathtools, dsfont}
\usepackage{tikz}

\usepackage{helvet}
\usepackage[margin=1in]{geometry}

\theoremstyle{definition}
\newtheorem{thm}{Theorem}[section]
\newtheorem{cor}[thm]{Corollary}
\newtheorem{defn}[thm]{Definition}
\newtheorem{example}[thm]{Example}

\newtheorem{lemma}[thm]{Lemma}
\newtheorem{prop}[thm]{Proposition}
\newtheorem{remark}[thm]{Remark}
\newtheorem{conj}[thm]{Conjecture}

\numberwithin{equation}{section}

\newcommand{\OO}{{\mathcal O}}

\newcommand{\Z}{\mathbb{Z}}

\newcommand{\Q}{\mathbb{Q}}
\newcommand{\R}{\mathbb{R}}
\newcommand{\C}{\mathbb{C}}

\renewcommand{\P}{\mathbb{P}}

\newcommand{\Gl}{{Gl}}
\newcommand{\Gr}{\operatorname{Gr}}
\newcommand{\Fl}{\operatorname{Fl}}
\newcommand {\cT}{\mathcal{T}}

\def\part{\partial}

\def\bd{\begin{defn}}
\def\ed{\end{defn}}
\def\bt{\begin{thm}}
\def\et{\end{thm}}
\def\br{\begin{remark}}
\def\er{\end{remark}}
\def\bc{\begin{cor}}
\def\ec{\end{cor}}
\def\bp{\begin{prop}}
\def\ep{\end{prop}}
\def\be{\begin{equation}}
\def\ee{\end{equation}}
\def\bn{\begin{enumerate}}
\def\en{\end{enumerate}}
\def\ba{\begin{array}}
\def\ea{\end{array}}
\def\bex{\begin{example}}
\def\eex{\end{example}}

\newcommand\sP{{\mathcal P}}

\newcommand\sF{{\mathcal F}}

\DeclareMathOperator{\rank}{Rank}

\title[Perverse sheaves on varieties with large fundamental groups]{Perverse sheaves on varieties with large fundamental groups}

\begin{document}

\author[D. Arapura ]{Donu Arapura}
\address{D. Arapura : Department of Mathematics, Purdue University, 150 N University St, West Lafayette, IN 47907, USA}
\email {arapura@math.purdue.edu}

\author[B. Wang]{Botong Wang}
\address{B. Wang : Department of Mathematics, University of Wisconsin-Madison, 480 Lincoln Drive, Madison WI 53706, USA}
\email {wang@math.wisc.edu}

\date{\today}


\date{\today}

\begin{abstract}
 We conjecture that any perverse sheaf on a compact aspherical
 K\"ahler manifold has non-negative Euler characteristic. This extends
 the Singer-Hopf conjecture in the K\"ahler setting. We verify the  stronger conjecture when the manifold $X$ has non-positive sectional curvature. We also show that the conjecture holds when $X$ is projective and in possession of  a faithful semi-simple rigid local system.
The first result is proved by expressing the Euler characteristic as an intersection number involving the characteristic cycle,
and then using the curvature conditions to deduce non-negativity. For the second result, we have that  the local system underlies a complex variation of Hodge structure. We then  deduce the desired inequality from the curvature properties of the image of the period map.
\end{abstract}

\maketitle

\section{Introduction} 

In the 1930s, Hopf conjectured that a closed $2d$-dimensional compact Riemannian
manifold $X$ with non-positive sectional curvature satisfies
$(-1)^d\chi(X)\geq 0$. Since a simply connected complete Riemannian
manifold with non-positive sectional curvature is contractible, a
natural generalization of the Hopf conjecture is the following.

\begin{conj}[Singer-Hopf]
Let $X$ be a closed manifold of dimension $2d$. If $X$ is aspherical,
i.e. if the universal cover of $X$ is contractible, then
\[
(-1)^d \chi(X)\geq 0.
\]
\end{conj}

In fact, Singer conjectured that under the above hypothesis, all the
$L^2$ Betti numbers of $X$ vanishes except in degree $d$. When
combined with Atiyah's $L^2$-index theorem, it implies the conjecture
stated above. 
When $X$ is a compact K\"ahler hyperbolic manifold, Gromov
\cite{Gr} proved Singer's conjecture on $L^2$ Betti numbers.

{Motivated by \cite[Conjectures 6.2 and 6.3]{LMWd},} we  propose a strengthening of the Singer-Hopf
conjecture for K\"ahler manifolds.

\begin{conj}\label{conj-perverse}
  Let $X$ be an aspherical compact K\"ahler manifold, then for
  any perverse sheaf $\sP$ on $X$, $\chi(X, \sP)\geq 0$. 
\end{conj}

{ 
\begin{remark}
Combining the above conjecture and \cite[Conjecture 6.3]{LMWd}, one may hope that if a compact K\"ahler manifold has large fundamental group, then it has nef cotangent bundle. However, this statement is proved to be false in \cite{W}.
\end{remark}
}

Since the intersection complex of any irreducible variety is a
perverse sheaf (up to translation), this is indeed a strengthening of the Singer-Hopf
conjecture in the compact K\"ahler setting:

\begin{lemma}\label{lemma_IH}
Let $X$ be a compact K\"ahler manifold such that for
  any perverse sheaf $\sP$ on $X$, $\chi(X, \sP)\geq 0$.  Then for any closed submanifold $Y\subset X$, $(-1)^{\dim Y}\chi(Y)\geq 0$. More generally, for any irreducible subvariety $Y\subset X$, $(-1)^{\dim Y}\chi_{IH}(Y)\geq 0$, where $\chi_{IH}$ denotes the intersection cohomology Euler characteristics,
\[
\chi_{IH}(Y)=\sum_{0\leq i\leq 2\dim Y}(-1)^i\dim {IH}^i(Y, \Q). 
\]
\end{lemma}

This stronger statement was conjectured by Yongqiang Liu, Laurentiu Maxim and
the second author \cite{LMWd} for projective manifolds. We refer to
that paper for further discussion and background.
In this paper, we give further evidence for conjecture
\ref{conj-perverse}. Our first result is as  follows.

\begin{thm}\label{thm_0}
 Let $X$ be a compact K\"ahler manifold. If $X$ has non-positive holomorphic bisectional
  curvature,  then for
  any perverse sheaf $\sP$ on $X$, $\chi(X, \sP)\geq 0$.  If $X$ has negative holomorphic bisectional
  curvature, or more generally if the cotangent bundle is ample, then for
  any nonzero perverse sheaf $\sP$, $\chi(X, \sP) >0$.
\end{thm}

We recall the defintion of holomorphic bisectional
curvature in Section \ref{section:thm0}.

\begin{cor}\label{cor:0}
  Let $X$ be a compact K\"ahler manifold. If $X$ has  non-positive  sectional
  curvature  then conjecture \ref{conj-perverse} holds for $X$. If 
  $X$ has negative sectional curvature, then $\chi(X, \sP) >0$ for every nonzero
  perverse sheaf $\sP$.
\end{cor}

As a special case of the first statement, 
we find that  the Singer-Hopf conjecture holds for non-positively curved compact K\"ahler manifolds. This  was originally
proved by Cao-Xavier \cite{CX} and  Jost-Zuo
\cite[Corollary 1]{JZ} independently,
 by extending Gromov's vanishing theorem for non-middle $L^2$-Betti numbers (i.e. Singer's conjecture) to such manifolds.
 When $X$ is a negatively curved compact K\"ahler manifold, Gromov  \cite{Gr} also established
the strict inequality $(-1)^{\dim X} \chi(X)> 0$, by proving a nonvanishing theorem for 
 the middle $L^2$-Betti number. Our proof is completely different, and both cases are handled  the same way. {The key idea is to use Kashiwara's index theorem, which first appeared in the proof of \cite[Proposition 3.3]{LMWd}.}

Before stating the next result, we  recall a few  definitions.
Given a finitely generated group $\Gamma$, let $\overline{\Gamma}$ denote the quotient of
$\Gamma$ by the intersection of all normal subgroups of finite index. The group  $\overline{\Gamma}$
can be equivalently defined as the largest residually finite quotient, or as 
the image of $\Gamma$ in its profinite completion. A well known theorem
of Malcev implies that any  representation $\rho:\Gamma\to GL(r,\C)$ factors through a representation
 $\overline{\rho}:\overline{\Gamma}\to GL(r,\C)$. Let us say that $\rho$ is {\em almost faithful} if $\overline{\rho}$
 is faithful.
Following Koll\'ar \cite[Chapter 4]{Ko}, a compact K\"ahler manifold
$X$ is said to have \emph{ large algebraic fundamental group}, if for any
connected positive-dimensional compact complex manifold $Y$ and any
non-constant map $Y\to X$,
the induced map $\overline{\pi_1(Y)}\to \overline{\pi_1(X)}$ has an infinite image. 
 Let $\Gamma$ be a discrete group and let $V$ be a finite-dimensional complex vector space. A representation $\rho:\Gamma\to GL(V)$ is called \emph{cohomologically rigid} if $H^1(\Gamma, End(V))=0$, or equivalently if it has no
 nontrivial first   order deformations. Our second theorem is the
 following.
 
 \begin{thm}\label{thm_2}
Let $X$ be a smooth complex projective variety with large algebraic fundamental
group. Suppose that there exists a cohomologically rigid almost faithful
semi-simple representation $\rho: \pi_1(X) \to \Gl(r,
\mathbb{C})$. Then, every perverse sheaf $\sP$ on $X$ satisfies that
$\chi(X, \sP)\geq 0$.
\end{thm}

We discuss some examples where this theorem applies in Section \ref{section:examples}.
If a smooth projective variety is aspherical, then it has large
fundamental group (see \cite[Proposition 6.7]{LMWd}), hence  large algebraic
fundamental group if the group is residually finite.
So we conclude that:

\begin{cor}
  Let $X$ be an aspherical projective manifold possessing a cohomologically rigid faithful semi-simple
local system, then conjecture  \ref{conj-perverse}  (and
in particular, the Singer-Hopf conjecture) holds for $X$.
\end{cor}

Theorem \ref{thm_2} is deduced from the following  theorem, which is
the technical heart of the paper.

\begin{thm}\label{thm_1}
Let $X$ be a compact K\"ahler manifold with large algebraic fundamental
group. Assume that there exists an almost faithful semi-simple representation
$\rho: \pi_1(X) \to \Gl(r, \C)$, with discrete image, such that the
corresponding local system $L_\rho$ underlies a complex variation of
Hodge structure (CVHS). Then, every perverse sheaf $\sP$ on $X$ satisfies that $\chi(X, \sP)\geq 0$.
\end{thm}

Let us sketch the ideas behind the proofs. By combining  theorems of  Simpson and
Esnault-Groechenig, we can  conclude that a cohomologically rigid local system on
a projective manifold comes from a CVHS with discrete monodromy.
This reduces Theorem \ref{thm_2} to  Theorem~\ref{thm_1}.
For the proof of Theorem \ref{thm_1}, we observe that there is a period
map $\phi:X\to \Gamma\backslash D$ to a quotient of a period domain.
The conditions of the theorem forces $\phi:X\to Z$ to be finite, {where $Z$ is the image of $\phi$}. Therefore
$\phi_*(\sP)$ is still perverse, and $\chi(\sP)=\chi(\phi_*(\sP))$. So
we can replace $(X,\sP)$ by $(Z, \phi_*(\sP))$. Suppose that $\Gamma\backslash D$ and image $Z=\phi(X)$ are both manifolds.
Then  $Z$ is a {\em
  horizontal} submanifold. Standard curvature calculations show that
$Z$ has a nef cotangent bundle. We can compute
$\chi(\phi_*(\sP))$ as the intersection number of the characteristic
cycle  $CC(\phi_*(\sP))\subset T^*Z$ and the zero section. The
non-negativity of this intersection number follows from a proposition
of Demailly-Peternell-Schneider. In general, extra steps are needed to
handle the possibilities that $ \Gamma\backslash D$ is only
an orbifold, and that $Z$ may be singular. In essence, the last issue
is handled by resolving singularities of $Z$ in such a way that the
cotangent bundle on the smooth locus extends to nef bundle.
The proof of theorem \ref{thm_0} is similar, although considerably
easier. In fact, it serves as a template for the proof just outlined.

\textbf{Acknowledgments.} We thank H\'el\`ene Esnault, Laurentiu
Maxim, and Sai-Kee Yeung for helpful discussions, and the referees for useful comments. The first author is  partially supported by
a collaboration grant from the Simons foundation. The second author is partially supported by a Sloan fellowship.

\section{Examples}\label{section:examples}

In this section, we discuss a few examples where Theorems \ref{thm_0}
and  \ref{thm_2} can be applied.

Let $Y=\Gamma\backslash G/K$ be a nonsingular Shimura variety, or more precisely a quotient of a Hermitian symmetric
domain $ G/K$ by a torsion free arithmetic group $\Gamma$. Note that we have
implicitly chosen a $\Q$-algebraic group $G_\Q$ such that $G$ is the
identity component of $G_\Q(\R)$. Let us assume that  $G_\Q$ is
simple, and its $\Q$-rank $r\in [2,\dim Y-2]$.
We have the Baily-Borel
compactification $\overline{Y}$, which is a (generally singular)
projective variety. Our assumption on $\Q$-rank implies that the boundary $\overline{Y}-Y$ has
codimension $\dim Y-r+1\ge 3$. 
Therefore we can find a projective manifold $X\subset Y$, with  $\dim X \in [2, \dim Y-r]$, obtained as an
intersection of  hyperplanes in general position. An appropriate Lefschetz theorem \cite{HL} 
shows that $\pi_1(X)=\Gamma$.
Any faithful representation $\rho$ of $G$ restricts to a faithful semisimple
representation of $\Gamma$. Raghunathan's vanishing theorem \cite{R}
shows that $\rho$ is cohomologically rigid. Finally, we note that
since $G/K$ is Stein and $\Gamma$ is residually finite,  $X$ can be seen to have large algebraic
fundamental group. Thus Theorem \ref{thm_2} applies to show that
perverse sheaves on $X$ have non-negative Euler characteristic. Note
$X$ inherits a metric from $Y$ with non-positive
curvature, so Theorem \ref{thm_0} also applies to
obtain the same conclusion. {When $Y$ is a quotient of a ball, the metric has negative bisectional curvature,
so  perverse sheaves have strictly positive  Euler characteristic.}

A construction of Catanese-Koll\'ar-Nori-Toledo produces another
example, where our  theorems can be applied. In this example, one sees
the need for allowing almost faithful representations.
Let $G_\Q=Sp(2g,\Q)$,
with $g\ge 3$. The associated symmetric space is the Siegel upper half plane $\mathbb{H}_g$.
Let $\Gamma\subset Sp(2g,\Z)$ be a torsion free arithmetic group. The
conditions of the last paragraph hold for 
$Y=\Gamma\backslash \mathbb{H}_g$, so we may construct a projective
manifold $X\subset Y$, as
above, such that $\pi_1(X)=\Gamma$. Note that $\dim X$ can be chosen
to be any integer in $ [2, \frac{1}{2}g(g+1)-g-1]$.
Fix an odd integer $n\ge 3$.
By the construction described in \cite[page
136-137,139]{CK}, one obtains a nonsingular finite branched cover
$\pi: X'\to X$ such that
$$0\to \Z/n\Z\to\pi_1( X')\xrightarrow{\pi_*}\Gamma\to 1$$
is exact. (In the reference, it is assumed that $\dim X=2$, but the construction generalizes
in an obvious way.)
Furthermore, $\overline{\pi_1(X')} = \Gamma$,
so the fundamental group is not residually finite. Let $\rho:\pi_1(X')\to
Gl(2g, \C)$ be the pull back of the standard representation of
$\Gamma$, and let $V_\rho= \C^{2g}$ denote the representation space.
The Hochschild-Serre spectral sequence plus Raghunathan's theorem
shows that
$$H^1\big(\pi_1(X'), End(V_\rho)\big)\cong H^1\big(\Gamma, End(V_\rho)\big)=0$$
Therefore $\rho$ is cohomologically rigid. Since $\pi$ is finite,
one can deduce that $X'$ has large algebraic
fundamental group. Thus we may apply Theorem \ref{thm_2} to $X'$
to conclude that 
perverse sheaves on it  have non-negative Euler characteristic.
In this example, the same conclusion can be obtained from Theorem
\ref{thm_0} applied to $X$, since $\pi$ is finite.

\section{Characteristic cycles}
Let $X$ be a complex manifold, and let $\sF$ be an object of the bounded
constructible derived category of $\C$-vector spaces $D^b_c(X)$. If the support of $\sF$ is compact, then the Euler characteristic of $\sF$,
\[
\chi(X, \sF)\coloneqq \sum_{i\in \Z}(-1)^i\dim H^i(X, \sF)
\]
is well-defined.  Furthermore, associated to $\sF\in D^b_c(X)$ is the 
characteristic cycle $CC(\sF)$, which is 
 a finite $\Z$-combination of conic Lagrangian subvarieties in the cotangent
 bundle $T^*X$. More precisely,
 there exist closed irreducible subvarieties $\{Z_j\}_{1\leq j\leq l}$ of $X$ contained in the support of $\sF$ and integers $n_j$, such that
\begin{equation}\label{eq_CC}
CC(\sF)=\sum_{1\leq j\leq l} n_j \cdot T_{Z_j}^*X,
\end{equation}
where $T_{Z_j}^*X$ is the closure $\overline{T_{Z_{j, sm}}^*X}$ of the
conormal bundle of smooth locus of $Z_j$ in $T^*X$.
When $\sF$ is a perverse sheaf corresponding to a holomonic $D$-module
$\mathcal{M}$ under the Riemann-Hilbert correspondence,
 $CC(\sF)$ coincides with the characteristic cycle of $\mathcal{M}$.
 We refer to the textbooks \cite{KS} and \cite{Di} for the
construction of  $CC(\sF)$ in general. The key point for us, is 
 the following Global Index Formula of Kashiwara:

\begin{thm}[\cite{Kas}, {\cite[Theorem 4.3.25]{Di}}]\label{thm_global}
If $\sF$ has compact support, then
\[
\chi(X, \sF)=CC(\sF)\cdot T^*_X X
\]
where the right hand side denotes the intersection number of the characteristic
cycle {$CC(\sF)$} and the zero section {$T^*_X X$} in the cotangent bundle $T^*X$. 
\end{thm}

{Even though $T^*X$ is not compact, the above intersection number is well-defined. This is because the zero section $T^*_X X$ is compact. In fact, we can consider the intersection as an element in the Chow group $A_0(T^*_X X)$ (see \cite[Example 4.1.8]{Fu} for more details).}
We  need the following result about characteristic cycles of perverse sheaves. 

\begin{prop}\cite[Corollary 5.2.24]{Di}\label{prop_perv}
When $\sF$ is a perverse sheaf,  $CC(\sF)$ is effective, that is, the coefficients $n_j$ in \eqref{eq_CC} are non-negative. 
\end{prop}

\section{Curvature and positivity}\label{section:thm0}

We quickly review various relevant notions of curvature and positivity, and
refer to \cite[Section 7.5]{Z} for most of the details.
Let $(E,h)$ be a  holomorphic vector bundle with a Hermitian metric on a compact K\"ahler manifold
$X$. Then the associated curvature tensor $\Theta$ can be viewed as a
$C^\infty$ family of multilinear maps
$$\Theta:E_p\times \bar E_p\times T_p\times \bar T_p\to \C,\quad  p\in
X$$
The bundle $(E,h)$ is Griffiths semipositive (respectively positive), if
$$\Theta(u,\bar u, v,\bar v)\ge 0 \quad (\text{resp.} >0)$$
for all $p$, and nonzero $u\in E_p, v\in T_p$.
Griffiths semipositve bundles are nef
 in the appropriate sense \cite[Theorem 1.12]{DPS}, and the positive
 ones are ample \cite[Theorem B]{G}. The notion of nefness used in \cite{DPS} agrees
 with usual one when $X$ is projective.
The curvature $\Theta^*$ of
 the dual bundle $(E^*,h^*)$ satisfies
 $$\Theta^*(w^*,v^*,u,y) = -\Theta(v,w,u,y) $$
 where $u^*,w^*$ are the duals determined by $h$.

 When $E=TX$ with $h$ induced
 by the K\"ahler metric, the holomorphic bisectional curvature in the direction of
 two nonzero tangent vectors $u,v$ is
 $$B(u,v)=  \frac{\Theta(u,\bar u, v,\bar v)}{||u||^2||v||^2}$$
 We say that it is non-positive (respectively negative) if this is always less than or equal
 to (respectively less than) zero. The bisectional curvature can be written as a positive linear combination
 of two sectional curvatures \cite[page 178]{Z}, therefore:

 \begin{lemma}\label{lemma:GK}
    If $X$ has non-positive (respectively negative) sectional
curvature, then it has non-positive (respectively negative) holomorphic bisectional
curvature.
 \end{lemma}
 
  Putting the above facts together also shows that: 

 \begin{lemma}\label{lemma:nef}
   The holomorphic bisectional curvature of $X$ is non-positive (respectively negative) if and only
   if $T^*X$ is Griffiths semipositive (respectively positive). Under
   these conditions, $T^*X$ is nef (respectively ample).
 \end{lemma}

 We can now prove the first theorem.

 \begin{proof}[Proof of Theorem \ref{thm_0}]
   Let $\sP$ be a perverse sheaf on $X$. By the discussion of the
   previous section,
   $$\chi(X, \sP)=CC(\sP)\cdot T^*_X X$$
   where  $CC(\sP)$ a non-negative linear combination of cycles of the
   form $T_Z^*X$ where $Z\subset X$ is a closed irreducible analytic
   space.
   When $X$ has non-positive holomorphic bisectional curvature,
   it suffices to prove that
   $$T^*_Z X\cdot T^*_X X\geq 0$$
   for any such $Z$. Since by the previous lemma $T^*X$ is nef, the required
   inequality follows from \cite[Proposition 2.3]{DPS}.

   Now suppose that $T^*X$ is ample. This would hold when $X$ has negative holomorphic bisectional curvature by the previous lemma.
    Then $X$ is projective.
   Therefore we can apply a theorem of Fulton-Lazarsfeld  \cite[Theorem 2.1]{FL} to obtain the strict inequality
   $$T^*_Z X\cdot T^*_X X> 0.$$
   This shows that $\chi(X, \sP)>0$.
 \end{proof}

 Corollary \ref{cor:0} follows immediately from the theorem and lemma \ref{lemma:GK}. 

\section{CVHS and the period map}

We recall some facts about variations of Hodge structures, and period domains.
Further details can be found \cite{CMP}, or the first two sections of \cite{P}.
Suppose that $V$ is a CVHS on $X$ with monodromy representation
$\rho:\pi_1(X)\to GL_n(\C)$. Associated to $V$ is a period map
$\tilde \phi: {\tilde X}\to  D$ from the universal cover of $X$ to a Griffiths
period domain. The domain $D$ is a homogeneous complex manifold with a natural homogeneous hermitian metric,
and a distinguished subbundle $T^hD\subset TD$ called the horizontal subbundle.
 The map $\tilde \phi$ is holomorphic and horizontal in Griffiths' sense,
which means that the image of the derivative 
$\phi_*: TX\to \phi^*T D$  lands in the pullback of $T^hD$.

Let us now assume that the  monodromy group $\Gamma=\rho(\pi_1(X))$ is discrete. 

\begin{lemma}\label{lemma_discrete}
There exists a finite unramified cover $\pi: X'\to X$ such that the monodromy group of $V'\coloneqq \pi^* V$ is discrete and torsion free.
If $\rho $ is faithful and semi-simple, then the monodromy representation of $V'$ is faithful and semi-simple as well.
\end{lemma}

\begin{proof}
 By a well known lemma of Selberg \cite[Lemma 8]{Se}, $\Gamma$ contains a torsion free subgroup of finite index. Let $\pi:X'\to X$
 be the corresponding covering space. Then $V'=\pi^*V$ will have the desired properties.
\end{proof}

Replacing $X$ and $V$ by $X'$ and $V'$ respectively, we can assume that $\Gamma$ is torsion free. 
Then
$M=\Gamma\backslash D$ is a manifold, and the horizontal subbundle
$T^hD$ descends to a subbundle $T^hM\subset TM$.
The map $\tilde \phi$ induces a map $\phi:X\to M$ which is  horizontal
in the sense that its derivative lies in $T^hM$.

\begin{defn}
An irreducible analytic subvariety $Z$ of $M$ is called
\emph{horizontal}, if the embedding $Z_{\textrm{sm}}\times_{M}D\to D$
from the preimage of the smooth locus of $Z$ to $D$ is horizontal in
the above sense.
\end{defn}

{ 
\begin{prop}
An irreducible analytic subvariety $Z$ of $M$ is horizontal if and only if the Zariski tangent space $T_p Z$ at any point $p\in Z$ is contained in the subspace $T^h_pM$.
\end{prop}
\begin{proof}
The natural map $(\Omega_Z^1)^\vee\to (T_M/T^h_M)|_Z$, given by the differential of the inclusion $Z\subset M$ followed by projection, 
vanishes on the  smooth locus $Z_{sm}$. Therefore it vanishes everywhere.
Since $T_pZ$ can be identified with the fiber of the tangent sheaf $(\Omega_Z^1)^\vee\otimes \C_p$ for any $p\in Z$, it must lie in $T^h_pM$.
\end{proof}
}
\begin{cor}\label{cor_horizontal}
If a subvariety $Z\subset M$ is horizontal, then any irreducible subvariety $Z'\subset Z$ is also a horizontal subvariety. 
\end{cor}
{ 
\begin{proof}
By the ``only if" part of the preceding proposition, we have $T_pZ\subset T_p^hM$ for any $p\in Z$. Since $T_pZ'\subset T_pZ$, the desired statement follows from the ``if" part of the preceding proposition. 
\end{proof}
}

By \cite[Corollary 1.8]{P}, any  subbundle of the horizontal
subbundle of $T^hM$, which is an abelian Lie subalgebra
of $TM$,
has {non-positive}
holomorphic bisectional curvature with respect to the induced metric. In particular, if
$Z\subset M$ is a compact horizontal submanifold of
dimension $d$, then the tangent bundle $TZ$ of $Z$ has {non-positive}
holomorphic bisectional curvature by \cite[Lemma 3.1]{P} and the
previous sentence. Therefore, the degree of the top Chern class of $T^*Z$ is {non-negative}. Hence,
\[
(-1)^d\,\chi(Z)=(-1)^d\int_Z c_d(TZ)=\int_Z c_d(T^*Z)\geq 0.
\]
To prove our main results, we need to extend the above inequality 
to singular subvarieties. 
As a remedy for the possible singularity of $Z$, we state
the following result implicitly contained in the paper of Brunebarbe, Klingler and Totaro  \cite[Section 3]{BKT}.
 We also outline the proof following the argument in [loc. cit.], since the details  will be needed below.

\begin{prop}\label{prop_BKT}
Let $Z$ be a closed horizontal subvariety of $M$. Then $Z$ has a resolution of singularities $\tilde Z\supset Z_{sm}$, such that
$T^*Z_{sm}$ extends to a nef vector bundle $\cT$.
\end{prop}

\begin{proof}[Sketch of proof]
Denote the dimensions of $M$ and $Z$ by $n$ and $d$ respectively.  Let $\mathrm{Gr}_{T^*M}(n-d)$ be the Grassmannian bundle over $M$ parametrizing a $(n-d)$-dimensional subspace of $T_x^*M$ at each point $x\in M$. Denote by $f_M: \mathrm{Gr}_{T^*M}(n-d)\to M$ the natural projection. There is a tautological rank $(n-d)$ subbundle of $f_M^*(T^*M)$ which we denote by $F_1$. 

On a smooth point $x\in Z_{sm}$, its conormal space defines a point in the fiber of $\mathrm{Gr}_{T^*M}(n-d)$ over $x$.
Thus, we have a   holomorphic map $Z_{sm}\to \mathrm{Gr}_{T^*M}(n-d)$.
Let $Z_1$ be the closure of the graph of this map in $Z\times \mathrm{Gr}_{T^*M}(n-d)$. It follows from  a theorem of Remmert-Stein \cite[p 169]{GR}, that $Z_1$ is analytic.
Let $\tilde Z$ be a resolution of singularities of $Z_1$. We denote the composition ${\tilde Z}\to Z_1\to  \mathrm{Gr}_{T^*M}(n-d)$ by $f_2$. 
We can assume that $f_2$ is an isomorphism over $Z_{sm}$. 
Let $\cT =f_2^*(f_M^*(T^*M)/F_1)$. By construction,
the restriction of $\cT$ to $Z_{sm}$ is
$T^*{Z_{sm}}$. Furthermore, $\cT$ inherits a hermitian metric from $M$, which extends the metric on  $T^*{Z_{sm}}$.
The curvature results of Peters quoted above, imply that  $\cT$  is Griffiths {semipositive} over $Z_{sm}$. Consequently, $\cT$ is Griffiths semipositive,
and therefore nef, on the whole of ${\tilde Z}$ by continuity.
\end{proof}

\section{Intersection numbers}
The goal of this section is to prove the following proposition.
\begin{prop}\label{prop_positive}
Let $M=\Gamma\backslash D$ ($\Gamma$ is discrete and torsion free) be
the quotient of a period domain associated to a CVHS as in the previous section. Let $Z$ be a compact horizontal analytic subspace of $M$. Then as intersection numbers  in $T^*M$,
\[
T^*_Z M\cdot T^*_M M\geq 0.
\]
\end{prop}

We first recall two results in intersection theory. 
Let $N$ be a compact complex manifold of dimension $n$, and let $F$ be a rank $r$ complex vector bundle on $N$. Let $\P(F)$ denote the projective bundle of lines in the fibers of $F$. Let $q: \P(F)\to N$ be the projection. Notice that $\mathcal{O}_{\P(F)}(-1)$ is naturally a subbundle of $q^*(F)$.
\begin{lemma}\label{lemma_1}\cite[Example 3.3.3]{Fu}
Let 
\[
\xi=c_{r-1}\big(q^*F/\mathcal{O}_{\P(F)}(-1)\big)\in H^{2n}\big(\P(F), \Q\big).
\]
Then, for any $\alpha\in H^*(N, \Q)$
\[
q_*(\xi\cup q^*\alpha)=\alpha,
\]
where $q_*$ is the Gysin homomorphism.
\end{lemma}

Let $C\subset F$ be an $r$-dimensional irreducible conic subvariety whose support on $N$ is compact. Thus, the associated projective variety $\P(C)$ is a compact closed subvariety of $\P(F)$. 
Let $i: \P(C)\to \P(F)$ be the closed embedding of the associated projective varieties. 
\begin{lemma}\label{lemma_2}\cite[Corollary 8.1.14]{La}
Under the above notations, 
\[
C\cdot \{\textrm{zero section of } F\}=\int_{\P(C)} i^*c_{r-1}\big(q^*F/\mathcal{O}_{\P(F)}(-1)\big)
\]
where the left-hand-side is the intersection number of subvarieties in $F$ and $c_{r-1}$ on the right-hand-side denotes the $(r-1)$-th Chern class. 
\end{lemma}
\begin{remark}
The statement in \cite[Corollary 8.1.14]{La} assumes that $N$ is projective. Nevertheless, the same proof still works under the weaker assumption {that $N$ is a complex manifold and the support of $C$ is compact (equivalently, the intersection of $C$ and the zero section of $F$ is compact). For more details, see also \cite[Definition 7.1 and Example 4.1.8]{Fu}.
}
\end{remark}

\begin{proof}[Proof of Proposition \ref{prop_positive}]
Let $\P(T^*M)$ be the projective bundle associated to the vector bundle $T^*M$, and let $p_M: \P(T^*M)\to M$ be the projection. Let $P\coloneqq p_M^*(T^*M)/\mathcal{O}_{\P(T^*M)}(-1)$ be the quotient bundle. By Lemma \ref{lemma_2}, we have
\begin{equation}\label{eq_41}
T^*_Z M\cdot T^*_M M=\int_{\P(T^*_Z M)} g^*c_{n-1}(P)
\end{equation}
where $g: \P(T^*_Z M)\to \P(T^*M)$ is the closed embedding.

Let $\Fl_{T^*M}(1, n-d)$ be the flag bundle over $M$ parametrizing  flag
$V_0\subset V_1\subset T_x^* M$ at each point $x\in M$, where $\dim V_0=1$ and $\dim V_1= n-d$.
Consider the following commutative diagram:
\begin{equation}\label{diagram_1}
\xymatrix{
\P(T^*_Z M)\ar[r]^{g}\ar[d]^{p_Z}& \P(T^*M)\ar[d]^{p_M}& \mathrm{Fl}_{T^*M}(1, n-d)\ar[l]_{g_M\quad}\ar[d]^{p_{\Gr}}\\
Z\ar[r]^{f}& M& \mathrm{Gr}_{T^*M}(n-d)\ar[l]_{f_M\quad}
}
\end{equation}
where  $f, g$ are the inclusion maps, all vertical arrows and $f_M$, $g_M$ are natural projections.

As discussed in the proof of Proposition \ref{prop_BKT}, there exists a
meromorphic map $Z\dashrightarrow \Gr_{T^*M}(n-d)$, which extends to a
morphism $f_2:\tilde Z\to \Gr_{T^*M}(n-d)$.
Recall that the tautological rank $(n-d)$ subbundle of $f_{M}^*(T^*M)$
is denoted by $F_1$. 
 Since the flag bundle can be defined as an iterated projective bundle, there is a natural isomorphism 
\[
\Fl_{T^*M}(1, n-d)\cong \P(F_1).
\]
We define $\tilde Y = \P(f_2^*F_1)$. Then we claim that the following
commutative diagram exists, containing \eqref{diagram_1}:
\begin{center}
\begin{tikzpicture}
\node (Y2) at (2, 4.2) {${\tilde Y}$};
\node (PZ) at (0, 3) {$\P(T^*_Z M)$};
\node (PM) at (4, 3) {$\P(T^*M)$};
\node (Fl) at (8, 3) {$\Fl_{T^*M}(1, n-d)$};
\node (Z) at (0, 0) {$Z$};
\node (M) at (4, 0) {$M$};
\node (Z2) at (2, 1.2) {${\tilde Z}$};
\node (Gr) at (8, 0) {$\Gr_{T^*M}(n-d)$};
\draw[->] (Y2) to ["$p_2$", swap] (Z2);
\draw[->] (PZ) to  (PM);
\draw[->] (Z) to  (M);
\draw[->] (PZ) to  (Z);
\draw[->] (PM) to  (M);
\draw[->] (Y2) to ["$g_1$", swap] (PZ);
\draw[->] (Z2) to ["$f_1$", swap] (Z);
\draw[->] (Fl) to  (PM);
\draw[->] (Gr) to  (M);
\draw[->] (Fl) to  (Gr);
\draw[->] (Y2) to ["$g_2$"] (Fl);
\draw[->] (Z2) to ["$f_2$"] (Gr);
\end{tikzpicture}
\end{center}
where $g_1$ remains to be constructed, and all other  arrows are the
natural maps. (Some labels are suppressed in this diagram for
readability.) To see that $g_1$ is defined, observe that the image
$g_M(g_2(\tilde Y))$ is an irreducible closed set containing $\P(T^*_{Z_{sm}}M)$.
By construction, it also follows that $g_1$ is bimeromorphic.

Since $ g_1$ is a bimeromorphic map, it maps the fundamental class of ${\tilde Y}$ to the fundamental class of $\P(T^*_Z M)$. Thus, equation \eqref{eq_41} can be continued as
\begin{equation}\label{eq_42}
\int_{\P(T^*_Z M)} g^*c_{n-1}(P)=\int_{{\tilde Y}} (g_M \circ g_2)^*c_{n-1}(P),
\end{equation}
and we recall that $P= p_M^*(T^*M)/\mathcal{O}_{\P(T^*M)}(-1)$. Notice that there exists a tautological flag of subbundles of $(p_M\circ g_M)^*(T^*M)$, which we denote by 
\[
E_0\subset E_1\subset (p_M\circ g_M)^*(T^*M)
\]
with $\rank(E_0)=1$ and $\rank(E_1)=n-d$. As subbundles of $(p_M\circ g_M)^*(T^*M)$, we have 
\[
g_M^*(\mathcal{O}_{\P(T^*M)}(-1))=E_0.
\]
Thus, we have a natural isomorphism
\begin{equation*}
g_M^*(P)\cong (p_M\circ g_M)^*(T^*M)/E_0.
\end{equation*}
Thus, equation \eqref{eq_42} can be continued as
\begin{equation}\label{eq_43}
\int_{{\tilde Y}} (g_M \circ g_2)^*c_{n-1}(P)=\int_{{\tilde Y}} g_2^*c_{n-1}(g_M^*(P))=\int_{{\tilde Y}} g_2^*c_{n-1}\big((p_M\circ g_M)^*(T^*M)/E_0\big).
\end{equation}
By the short exact sequence of vector bundles on $\Fl_{T^*M}(1, n-d)$
\[
0\to E_1/E_0\to (p_M\circ g_M)^*(T^*M)/E_0\to (p_M\circ g_M)^*(T^*M)/E_1\to 0,
\]
we have
\begin{equation}\label{eq_e1}
c_{n-1}\big((p_M\circ g_M)^*(T^*M)/E_0\big)=c_{n-d-1}(E_1/E_0)\cup c_{d}\big((p_M\circ g_M)^*(T^*M)/E_1\big). 
\end{equation}

 By definition, as subbundles of $(p_M\circ g_M)^*(T^*M)$, we have $p_{\Gr}^*(F_1)=E_1$. Hence
\begin{equation}\label{eq_e2}
(p_M\circ g_M)^*(T^*M)/E_1\cong p_{\Gr}^*\big(f_M^*(T^*M)/F_1\big).
\end{equation}
Moreover, 
\begin{equation}\label{eq_e3}
g_2^*(E_1)=g_2^*p_{\Gr}^*(F_1)=p_2^*f_2^*(F_1).
\end{equation}
Moreover, as subbundles of \eqref{eq_e3}, we have 
\begin{equation}\label{eq_e4}
g_2^*(E_0)=\mathcal{O}_{\P(f_2^*(F_1))}(-1).
\end{equation}
Now, we have
\begin{equation}\label{eq_44}
\begin{split}
\int_{{\tilde Y}} g_2^*c_{n-1}&\big((p_M\circ g_M)^*(T^*M)/E_0\big)=\int_{{\tilde Z}} p_{2*}g_2^*c_{n-1}\big((p_M\circ g_M)^*(T^*M)/E_0\big)\\
&=\int_{{\tilde Z}} p_{2*}\Big(g_2^*c_{n-d-1}(E_1/E_0)\cup g_2^* c_{d}\big((p_M\circ g_M)^*(T^*M)/E_1\big)\Big)\\
&=\int_{{\tilde Z}} p_{2*}\Big(c_{n-d-1}(g_2^*E_1/g_2^*E_0)\cup  c_{d}\big(g_2^*p_{\Gr}^*\big(f_M^*(T^*M)/F_1\big)\big)\Big)\\
&=\int_{{\tilde Z}} p_{2*}\Big(c_{n-d-1}\big(p_2^*f_2^*(F_1)/\mathcal{O}_{\P(f_2^*(F_1))}(-1)\big)\cup  c_{d}\big(p_2^*f_2^*\big(f_M^*(T^*M)/F_1\big)\big)\Big)\\
&=\int_{{\tilde Z}} p_{2*}\Big(c_{n-d-1}\big(p_2^*f_2^*(F_1)/\mathcal{O}_{\P(f_2^*(F_1))}(-1)\big)\cup  p_2^*c_{d}\big(f_2^*\big(f_M^*(T^*M)/F_1\big)\big)\Big)\\
&=\int_{{\tilde Z}} c_{d}\big(f_2^*\big(f_M^*(T^*M)/F_1\big)\big)
\end{split}
\end{equation}
where the second equality follows from \eqref{eq_e1}, the third follows from \eqref{eq_e2}, the fourth follows from \eqref{eq_e3} and \eqref{eq_e4}, and the last one follows from Lemma \ref{lemma_1}.

Notice that the nef vector bundle $\cT$ in Proposition \ref{prop_BKT} is equal to $f_2^*(f_M^*(T^*M)/F_1)$. Thus, by \cite[Corollary 2.2]{DPS}, we have
\begin{equation}\label{eq_45}
\int_{{\tilde Z}} c_{d}\big(f_2^*\big(f_M^*(T^*M)/F_1\big)\big)\geq 0.
\end{equation}
Finally, the inequality of the proposition follows from \eqref{eq_41}, \eqref{eq_42}, \eqref{eq_43}, \eqref{eq_44} and \eqref{eq_45}. 
\end{proof}

\section{Proof of the remaining theorems}
\begin{proof}[Proof of Theorem \ref{thm_1}]
{
Since $\pi: X'\to X$ is a finite covering map and since being perverse sheaf is a local property, for any perverse sheaf $\sP$ on $X$, $\pi^*\sP$ is also a perverse sheaf. 
Moreover, we claim that
\[\chi(X', \pi^*\sP)=\deg(\pi)\cdot \chi(X, \sP).
\]
In fact, by the additivity of Euler characteristics on distinguished triangles, it suffices to show the above equality with $\sP$ replaced by a local system (see e.g., \cite[Theorem 4.1.22]{Di}). Since the Euler characteristic of a local system is equal to the product of the rank and the Euler characteristic of the underlying space, the formula further reduces to the standard formula of the Euler characteristic of finite covering space. }
Therefore, it suffices to show that $\chi(X', \sP')\geq 0$ for any perverse sheaf on $X'$. Replacing $X$ by $X'$ in Lemma \ref{lemma_discrete}, we can assume that the monodromy group of $L_\rho$ is torsion free.

Let $\phi: X\to M= \Gamma\backslash D$ be the period map associated to
the CVHS on $L_\rho$. 
We claim that $\phi$ is quasi-finite, that is, the preimage of every
point is a finite set. Suppose $Y\subset X$ is an irreducible closed
subvariety such that $\phi(Y)$ is a point. Let $Y'\to Y$ be a
resolution of singularity. Since $X$ has large algebraic fundamental group and $\overline{\rho}$ is a faithful representation, the pullback of $L_\rho$ to $Y'$ underlies a nontrivial CVHS. Thus, the period map associated to $L_\rho$ is nontrivial, a contradiction to $\phi(Y)$ being a point.

Since $X$ is compact and $\phi$ is quasi-finite, $\phi$ is indeed a finite morphism. Therefore, $\phi_*(\sP)$ is a perverse sheaf on $M$, whose support is contained in the compact subvariety $\phi(X)$. Since $\chi(X, \sP)=\chi(M, \phi_*(\sP))$, it suffices to show that $\chi(M, \phi_*(\sP))\geq 0$. 

Let 
\[
CC\big(\phi_*(\sP)\big)=\sum_{1\leq j\leq l} n_j \cdot T_{Z_j}^* M.
\]
By Proposition \ref{prop_perv}, $n_j\geq 0$ for every $j$. By Theorem \ref{thm_global}, it suffices to show that as intersection numbers in $T^*M$, 
\begin{equation}\label{eq_last}
T_{Z_j}^* M\cdot T^*_M M\geq 0.
\end{equation}
By Griffiths transversality, $\phi(X)$ is a horizontal subvariety of $M$. By {Corollary \ref{cor_horizontal},} $Z_j$ is also a horizontal subvariety of $M$. Therefore, the inequality \eqref{eq_last} follows from Proposition \ref{prop_positive}. 
\end{proof}

\begin{proof}[Proof of Theorem \ref{thm_2}]
 By a theorem of Esnault and Groechenig \cite{EG}, we can assume that
  the image of $\rho$ lies in $Aut(P)$, where $P$ is projective
  $\OO_K$-module of rank $n$, and $\OO_K$ is the ring
  of integers of a number field $K$. Now, we apply Lemma \ref{lemma_discrete}, and the monodromy group of $\pi^*(L_\rho)$ is also contained in $Aut(P)$. Replacing $X$ and $L_\rho$ by $X'$ and $\pi^*L_\rho$ respectively, we can assume that the monodromy group is torsion free by {\cite[Lemma 8]{Se}}.

  Let $r=[K:\Q]$ and let $\sigma_i:K\to \C, i=1,\ldots
  r$, be the complete set of complex embeddings up to equivalence. Since $P\otimes_{\OO_K}\C\cong \C^n$, extending the scalars of $\rho$ by $\sigma_i$ induces a complex representation $\rho_i: \pi_1(X)\to GL(n, \C)$. 
Then
\[
\tilde \rho\coloneqq \prod_{i=1, \ldots, r}\rho_i :\pi_1(X)\to  \prod_{i=1, \ldots, r} GL(n, \C)\subset GL(rn, \C)
\]
has discrete image. 
{ On the other hand, cohomological rigidity implies that $\tilde \rho$ is an isolated point of the Betti moduli space or representation variety.
Hence, under the correspondence provided by the nonabelian Hodge theorem \cite{S}, the corresponding Higgs bundle is an isolated point of the Dolbeault moduli space.
So, by a result of Simpson \cite[Corollary  4.2]{S}, $\tilde \rho$ is the monodromy representation of a CVHS.}
  Therefore, Theorem \ref{thm_1} and Lemma \ref{lemma_IH} apply.
\end{proof}


\end{document}